\documentclass[11pt,letterpaper]{amsart}
\usepackage{amsfonts, amsmath, amssymb, amscd, amsthm, graphicx,enumerate}
\usepackage{epsfig, color}
\usepackage{hyperref}

\makeatletter
\def\blfootnote{\xdef\@thefnmark{}\@footnotetext}
\makeatother

\newtheorem{thm}{Theorem}[section]

\newtheorem{conj}[thm]{Conjecture}
\theoremstyle{definition}

\theoremstyle{remark}
\newtheorem{rem}[thm]{Remark}

\newfont{\eufm}{eufm10}

\renewcommand{\phi}{\varphi}

\begin{document}

\title{A finitely generated group that does not satisfy the generalized Burghelea Conjecture}

\author{A. Dranishnikov and M. Hull}

\date{}



\maketitle

\begin{abstract}
We construct a finitely generated group that does not satisfy the generalized Burghelea conjecture.
\end{abstract}


\section{Generalized Burghelea Conjecture}

In ~\cite{Bu}, Burghelea gave an explicit formula for the periodic cyclic homology of groups algebras with rational coefficients (and more generally with coefficients in fields of characteristic zero):

$$
PHC_\ast(\mathbb Q G)=\bigoplus_{[x]\in{\langle G\rangle}^{fin},n\ge 0}H_{2n+\ast}(N_x,\mathbb Q)\oplus \bigoplus_{[x]\in{\langle G\rangle}^{\infty}}T_\ast(x;\mathbb Q)
$$
where the group $T_*(x;\mathbb Q)=\lim_{\leftarrow}\{H_{*+2n}(N_x,\mathbb Q)\}$.

Here $G_x$ denoes the the centralizer of $x$ in $G$, $N_x=G_x/\langle x\rangle$ is the reduced centralizer, ${\langle G\rangle}^{fin}$ is the set of conjugacy classes of elements of finite order, and ${\langle G\rangle}^{\infty}$ the set of conjugacy classes of elements of infinite order. The bonding maps in the inverse sequences are the Gysin homomorphisms $S: H_{m+2}(N_x,\mathbb Q)\to H_{m}(N_x,\mathbb Q)$ corresponding to the fibration $B\langle x\rangle\simeq S^1\to BG_x\to BN_x$.
\begin{conj}[Generalized Burghelea Conjecture]
Let $G$ be a discrete group, then $T_*(x;\mathbb Q)=0$ for all $x\in{\langle G\rangle}^{\infty}$.
\end{conj}
Burghelea stated the above conjecture for groups $G$ which admit a finite $K(G, 1)$ ~\cite{Bu}. In the same paper Burghelea constructed a countable group that does not satisfy the Generalized Burghelea Conjecture. There are still no known counterexamples to the original version of Burgelea's conjecture, and it is known to hold for many classes of groups \cite{D, EM}.


The following is our main result.

\begin{thm}\label{main}
There is a finitely generated group $G$ that does not satisfy the Generalized Burghelea Conjecture.
\end{thm}

Our strategy is to show that the countable group constructed by Burghelea can be embedded in a finitely generated group in a way that preserves centralizers. This embedding is based on the theory of small cancellation over relatively hyperbolic groups developed by Osin \cite{Osi10}.

\begin{rem}
Shortly after this paper was written the authors became aware that Engel and Marcinkowski also had constructions of finitely generated counterexamples to the generalized Burghelea Conjecture, including a finitely presented counterexample and a counterexample of type $F_\infty$, using completely different methods. These constructions have since been added to \cite{EM}.  
\end{rem}

\section{Malnormal Embeddings} For a torsion-free group $G$ hyperbolic relative to a subgroup $H$, a subgroup $S$ of $G$ is called \emph{suitable} if $S$ contains two infinite order elements $f$ and $g$ which are not conjugate to any elements of $H$ and such that no non-trivial power of $f$ is conjugate to a non-trivial power of $g$. This is equivalent to \cite[Definition 2.2]{Osi10} since $G$ is torsion-free. Indeed, the maximal virtually cylic subgroups containing $f$ and $g$ respectively are both cyclic, and since no power of $f$ is equal to a power of $g$ these cyclic subgroups must intersect trivially.

The following is a special case of \cite[Theorem 2.4]{Osi10}.
\begin{thm}\label{thm:scq}
Let $G$ be a torsion-free group hyperbolic relative to a subgroup $H$, let $t\in G$, and let $S$ be a suitable subgroup of $G$. Then there exists a group $\overline{G}$ and an epimorphism $\gamma\colon G\to\overline{G}$ such that:
\begin{enumerate}
\item $\gamma|_H$ is injective (so we identify $H$ with its image in $\overline{G}$).
\item $\overline{G}$ is hyperbolic relative to $H$.
\item $\gamma(t)\in\gamma(S)$.
\item $\gamma(S)$ is a suitable subgroup of $\overline{G}$.
\item $\overline{G}$ is torsion-free.
\end{enumerate}
\end{thm}

We inductively apply the previous theorem to construct the desired embedding. This can be extracted from the proof of \cite[Theorem 2.6]{Osi10}, but since it is not explicitly stated there we include the proof below.

Recall that a subgroup $H$ of a group $G$ is called \emph{malnormal} if for all $x\in G\setminus H$, $x^{-1}Hx\cap H=\{1\}$.
\begin{thm}\label{malnormal}
Let $H$ be a torsion-free countable group. Then there exists a finitely generated group $\Gamma$ which contains $H$ as a malnormal subgroup.
\end{thm}


\begin{proof}
Let $H=\{1=h_0, h_1, h_2,...\}$. We inductively define a sequence of quotients as follows: Let $G_0=H\ast F$, where $F=F(x, y)$ is the free group  on $\{x, y\}$. Then $G_0$ is torsion-free, hyperbolic relative to $H$, and $F$ is a suitable subgroup of $G_0$. Let $\alpha_0\colon G_0\to G_0$ be the identity map. Suppose now we have constructed a torsion-free group $G_i$ together with an epimorphism $\alpha_i\colon G_0\to G_i$ such that:
\begin{enumerate}
\item $\alpha_i|_H$ is injective (so we identify $H$ with its image in $G_i$)
\item $G_i$ is hyperbolic relative to $H$.
\item $\alpha_i(h_j)\in\alpha_i(F)$ for all $0\leq j\leq i$.
\item $\alpha_i(F)$ is a suitable subgroup of $G_i$.
\item $G_{i}$ is torsion-free.
\end{enumerate}

Given such $G_i$, we can apply Theorem \ref{thm:scq} to $G_i$, $H$, $t=h_{i+1}$, and $S=\alpha_i(F)$. Let $G_{i+1}=\overline{G_i}$ be the quotient provided by Theorem \ref{thm:scq}. Define $\alpha_{i+1}=\gamma\circ\alpha_i$, where $\gamma$ is the epimorphism given by Theorem \ref{thm:scq}. Then Theorem \ref{thm:scq} implies that $\alpha_i\colon G_i\to G_{i+1}$ satisfies conditions (1)--(5).

Let $\Gamma$ be the direct limit of the sequence $G_0\to G_1\to G_2...$, that is $\Gamma=G_0/\bigcup\ker(\alpha_i)$. Let $\beta\colon G_0\to \Gamma$ be the natural quotient map. Note that $\beta|_F$ is surjective by construction. Indeed, $G_0$ is generated by $H\cup\{x, y\}$ and for each $h_i\in H$, $\alpha_i(h_i)\in \alpha_i(F)$, hence $\beta(h_i)\in\beta(F)$. Thus $\Gamma$ is generated by $\{\beta(x), \beta(y)\}$.

Now $\beta|_H$ is injective, so $H$ embeds in $\Gamma$; we identify $H$ with its image in $\Gamma$. Suppose $x\in \Gamma$ such that $x^{-1}Hx\cap H\neq\{1\}$. Then there exist $g, h\in H\setminus\{1\}$ such that $x^{-1}gxh^{-1}=1$. Let $\tilde{x}\in G_0$ such that $\beta(\tilde{x})=x$. Then for some $i\geq 1$, $\tilde{x}^{-1}g\tilde{x}h^{-1}\in\ker\alpha_i$. This means that $\alpha_i(\tilde{x})^{-1}H\alpha_i(\tilde{x})\cap H\neq\{1\}$. Since $G_i$ is hyperbolic relative to $H$, $H$ is malnormal in $G_i$ by \cite[Lemma 8.3b]{Osi10}. Hence $\alpha_i(\tilde{x})\in H$, which means that $x=\beta(\tilde{x})\in H$. Therefore $H$ is malnormal in $\Gamma$. 
\end{proof}

\section{Proof of Theorem \ref{main}}

\begin{proof}[Proof of Theorem \ref{main}]
We start by reviewing the counterexample constructed by Burghelea. By the Kan-Thurston theorem \cite{BDH, KT}, there exists a group $G$ and a map $t\colon K(G, 1)\to\mathbb CP^\infty$ which induces an isomorphism on homology and cohomology. Burghelea observes that the group $G$ can be chosen to be torsion-free. The idea behind this observation is that since $\mathbb CP^\infty=\bigcup \mathbb CP^n$, $K(G, 1)$ can be constucted inductively as a union of the form $K(G, 1)=\bigcup K(G_i, 1)$, where $K(G_i, 1)=t^{-1}(\mathbb CP^n)$ and each $K(G_i, 1)$ is a finite CW-complex (see, for example, the proof in~\cite{Ma} of the Kan-Thurston theorem). Since $G=\bigcup G_i$ and each $G_i$ is torsion-free, $G$ is also torsion-free.


Note that $H_{2n}(\mathbb C P^\infty;\mathbb Q)\cong \mathbb Q$ and the Gysin homomorphism $S: H_{2n+2}(\mathbb C P^\infty;\mathbb Q)\to H_{2n}(\mathbb C P^\infty;\mathbb Q)$  for the canonical $S^1$-bundle $S^\infty\to \mathbb C P^\infty$ is an isomorphism, hence $\lim_{\leftarrow}\{H_{2n}(\mathbb C P^\infty;\mathbb Q),  S\}\cong\mathbb Q$.

 Let $$1\to\mathbb Z=\langle x\rangle\to H\to G\to 1$$ be the central extension extension that corresponds to a generator $a\in H^*(G)=\mathbb Z[a]$, $deg(a)=2$. Note that $H\cong\pi_1(Y)$, where $Y$ is the pull-back of the bundle $S^\infty\to \mathbb CP^\infty$ along $t$. Hence $N_x=H/\langle x\rangle\cong G$, and $T_0(x;\mathbb Q)\cong\lim_{\leftarrow}\{H_{2n}(G, \mathbb Q)\}\cong\lim_{\leftarrow}\{H_{2n}(\mathbb C P^\infty;\mathbb Q),  S\}\cong\mathbb Q$. $H$ is the group constructed by Burghelea.

We now apply Theorem~\ref{malnormal} to obtain a malnormal embedding $H\to \Gamma$ into a finitely generated group $\Gamma$, and we identify $H$ with its image in $\Gamma$. Since $H$ is malnormal, no elements of $\Gamma\setminus H$ will centralize $x$. Hence $\Gamma_x= H$ and $N_x\cong G$. Then as before, we get that $T_0(x;\mathbb Q)\cong\mathbb Q\ne 0$ in the group $\Gamma$.
\end{proof}


\vspace{1cm}

\end{document}